\documentclass[12pt]{amsart}

\usepackage{amsmath,amssymb,amsthm}

\usepackage{microtype}

\usepackage{mathtools}

\usepackage{enumitem}

\usepackage[numbers,sort&compress]{natbib}

\usepackage{hyperref}

\setlist[enumerate]{label=\upshape(\roman*)}

\theoremstyle{plain}

\newtheorem{theorem}{Theorem}[section]

\newtheorem{lemma}[theorem]{Lemma}

\newtheorem{proposition}[theorem]{Proposition}

\newtheorem{corollary}[theorem]{Corollary}

\theoremstyle{definition}

\theoremstyle{remark}

\newtheorem{remark}[theorem]{Remark}

\counterwithin{equation}{section}

\allowdisplaybreaks

\newcommand{\N}{\mathbb N}

\newcommand{\Z}{\mathbb Z}

\newcommand{\E}{\mathbb E}

\newcommand{\Prob}{\mathbb P}

\newcommand{\Fcal}{\mathcal F}

\newcommand{\Wcal}{\mathcal W}

\newcommand{\logplus}{\log^{+}}

\DeclareMathOperator{\Var}{Var}

\hypersetup{pdftitle={Uniform lower bounds for the average length of finite Pierce expansions},pdfauthor={Min Woong Ahn}}

\begin{document}

\title[Average length of finite Pierce expansions]{Uniform lower bounds for the average length of finite Pierce expansions}

\author{Min Woong Ahn}

\address{Department of Mathematics Education, Dankook University, 152, Jukjeon-ro, Suji-gu, Yongin-si, Gyeonggi-do 16890, Republic of Korea}

\email{mwahn@dankook.ac.kr}

\date{October 19, 2026}



\begin{abstract}

For integers $1\leq a\leq b$, the length of the finite Pierce expansion of $a/b$ is the number of
steps of the iteration $r\mapsto b\bmod r$ which starts at $a$ and ends at zero. Erd\H{o}s and
Shallit (1991) proved that the average of these lengths over $1\leq a\leq b$ is bounded below by a
positive constant times $\log\log b$, and suggested that it is at least of order $\log b$. In this
paper, we prove that the average length is at least $\sqrt{2\log b}-3(\log b)^{1/4}-11$ for every
$b\geq2$. We also show that, for each $0<c<\sqrt2$, the proportion of
numerators $a$ for which the length is less than $c\sqrt{\log b}$ tends to zero exponentially fast in
$\sqrt{\log b}$, with an explicit rate. Our approach refines the interval-counting argument of
Erd\H{o}s and Shallit by restricting the product of the initial digits. As applications, we obtain
similar bounds for reduced fractions and for numerators in a fixed congruence class.

\end{abstract}

\maketitle

\tableofcontents

\section{Introduction}\label{sec:introduction}

As is well known, the Pierce expansion expresses $x\in(0,1]$ in the form
\begin{equation}\label{eq:intro-Pierce}
x=\sum_{k\in\N}\frac{(-1)^{k+1}}{d_1(x)\cdots d_k(x)}
=\frac1{d_1(x)}-\frac1{d_1(x)d_2(x)}+\frac1{d_1(x)d_2(x)d_3(x)}-\cdots
\end{equation}
with the convention $1/\infty=0$, where the digit sequence $(d_k(x))_{k\in\N}$ is defined in
Section~\ref{sec:preliminaries}. If $x$ is irrational, then $(d_k(x))_{k\in\N}$ is a strictly
increasing sequence of positive integers. In the case where $x$ is rational, the digit sequence
consists of finitely many strictly increasing positive integers, followed by $\infty$'s. For further
basic notions and results on Pierce expansions, we refer the reader to~\cite{Ahn23,Pie29,Sha86}.

For integers $1\leq a\leq b$, the finite Pierce expansion of $a/b$ is closely related to the
remainder iteration
\begin{equation}\label{eq:intro-remainders}
r_0\coloneqq a,\qquad r_{j+1}\coloneqq b\bmod r_j\quad(r_j>0).
\end{equation}
Since every positive remainder is smaller than its predecessor, the iteration terminates after
finitely many steps. Let $\ell(a,b)$ denote the least $j\in\N$ such that $r_j=0$; that is, the final
division, which produces zero, is also counted. By Lemma~\ref{lem:algorithm} below, $\ell(a,b)$ is
equal to the number of finite digits of $a/b$, i.e., the length of the finite Pierce expansion of
$a/b$. Following Erd\H{o}s and Shallit~\cite{ES91}, we include the numerator $a=b$, for which
$\ell(b,b)=1$, and consider the average
\begin{equation}\label{eq:intro-average}
Q(b)\coloneqq\frac1b\sum_{a=1}^{b}\ell(a,b).
\end{equation}

The iteration~\eqref{eq:intro-remainders} appears in Shallit~\cite{Sha86} and as Algorithm~6 in
Mays~\cite{Mays87}. Erd\H{o}s and Shallit~\cite[Observation~1 and p.~45]{ES91} noted that its
relation to Pierce expansions is analogous to the relation between the Euclidean algorithm and
continued fractions. It should be pointed out, however, that $b$ remains fixed throughout the
iteration.

For the maximal length, it is known that
\[
c_0\frac{\log b}{\log\log b}\leq\max_{1\leq a\leq b}\ell(a,b)\leq C_\varepsilon b^{1/3-2/177+\varepsilon}
\]
for each $0<c_0<1$, each $\varepsilon>0$, and all sufficiently large $b$, where $C_\varepsilon>0$
depends only on $\varepsilon$. The upper bound with the exponent $1/3+\varepsilon$ is due to
Erd\H{o}s and Shallit~\cite[Theorem~2]{ES91}. Ke\v{s}elj~\cite[Theorem~2]{Kes96} refined this to
$(\sqrt[3]{18}+o(1))\,b^{1/3}(\log b)^{1/3}$ as $b\to\infty$, and Chase and
Pandey~\cite[Theorem~1.1]{CP22} improved the exponent. The lower bound was proved by
Ke\v{s}elj~\cite[Theorem~3]{Kes96}; a bound of the same order was later proved by Chase and
Pandey~\cite[Theorem~1.2]{CP22} (see also~\cite[Theorem~1.5]{BGJK26}). In addition,
Erd\H{o}s and Shallit~\cite[Theorem~6]{ES91} showed that $\max_{1\leq a\leq b}\ell(a,b)>c\log b$
for some absolute constant $c>0$ and infinitely many $b$. In the other direction,
Shallit~\cite[p.~38]{Sha86} suggested that $\max_{1\leq a\leq b}\ell(a,b)=O(\log b)$, and
Erd\H{o}s and Shallit~\cite[p.~44]{ES91} believed that it is $O((\log b)^2)$. Recently, Baraskar and
Vukusic~\cite{BV26} and Baraskar, Gokhale, Jain, and Kie\.{z}un~\cite{BGJK26} studied iterated sets
of remainders, which are closely related to the maximal length.

In this paper, our main focus is on the average~\eqref{eq:intro-average}. Erd\H{o}s and
Shallit~\cite[Theorem~7]{ES91} proved that $Q(b)$ is bounded below by a positive constant times
$\log\log b$. They also suggested that $Q(b)\geq c\log b$ for some absolute constant $c>0$ and all
sufficiently large $b$~\cite[p.~44]{ES91}. Our first main result improves the lower bound of
Erd\H{o}s and Shallit as follows.

\begin{theorem}\label{thm:mean-grid}

For every integer $b\geq2$, we have
\begin{equation}\label{eq:main-mean}
Q(b)\geq\sqrt{2\log b}-3(\log b)^{1/4}-11.
\end{equation}
In particular,
\begin{equation}\label{eq:mean-liminf}
\liminf_{b\to\infty}\frac{Q(b)}{\sqrt{\log b}}\geq\sqrt2.
\end{equation}

\end{theorem}

The constant $\sqrt2$ in~\eqref{eq:mean-liminf} can be explained by the law of large numbers for
the Pierce expansion digits due to Shallit~\cite[Theorem~16]{Sha86}, which states that
$(d_n(x))^{1/n}\to e$ as $n\to\infty$ for Lebesgue-almost every $x\in[0,1]$. For such $x$, we have
$\log(d_1(x)\cdots d_n(x))/n^2\to1/2$ as $n\to\infty$, and so the largest $n$ with
$d_1(x)\cdots d_n(x)\leq b$ is asymptotic to $\sqrt{2\log b}$ as $b\to\infty$. Since our argument
only uses initial blocks of digits whose product does not exceed $b$, the constant $\sqrt2$ reflects
a limitation of our method. In fact, the lower bound of order $\log b$ suggested by Erd\H{o}s and
Shallit, which remains open, would imply that the lower limit in~\eqref{eq:mean-liminf} is
infinite.

Our second main result shows that, for each $0<c<\sqrt2$ and each sufficiently large $b$, all but a
small proportion of the numerators $a$ satisfy $\ell(a,b)\geq c\sqrt{\log b}$. For each
$0<c<\sqrt2$, we define
\begin{equation}\label{eq:kappa}
\kappa_c\coloneqq\frac{(1-c^2/2)^2}{8\bigl((c+1)^3+(c+1)\bigr)}.
\end{equation}
Note that $0<\kappa_c<1/16$.

\begin{theorem}\label{thm:short-grid}

Let $0<c<\sqrt2$. Then, there exists an integer $b_c\geq2$ such that
\begin{equation}\label{eq:short-grid}
\#\{1\leq a\leq b:\ell(a,b)<c\sqrt{\log b}\}
\leq3b\exp\bigl(-\kappa_c\sqrt{\log b}\bigr)
\end{equation}
for every integer $b\geq b_c$.

\end{theorem}

In fact, Theorems~\ref{thm:mean-grid} and~\ref{thm:short-grid} are special cases of the following
more general result, in which the fractions $a/b$, $1\leq a\leq b$, are replaced by a finite
multiset of rational numbers with small discrepancy. For a rational number $x\in(0,1]$, let $\ell(x)$
denote the length of its finite Pierce expansion, so that $\ell(a/b)=\ell(a,b)$ by
Lemma~\ref{lem:algorithm}. Here and throughout, cardinalities and sums over a finite multiset are
counted with multiplicity.

\begin{theorem}\label{thm:discrepancy}

Let $\Fcal$ be a finite multiset of rational numbers in $(0,1]$ with cardinality $N\geq1$, and let
$\Delta>0$ satisfy
\begin{equation}\label{eq:discrepancy}
\bigl|\#(\Fcal\cap(u,v))-N(v-u)\bigr|\leq\Delta
\qquad(0\leq u<v\leq1).
\end{equation}
Define $R\coloneqq N/\Delta$ and $L\coloneqq\logplus R$, where $\logplus u\coloneqq\max\{0,\log u\}$
for $u>0$. Then, the following hold.
\begin{enumerate}
\item We have
\begin{equation}\label{eq:general-mean}
\frac1N\sum_{x\in\Fcal}\ell(x)
\geq\sqrt{2L}-3L^{1/4}-11.
\end{equation}
\item For each $0<c<\sqrt2$, there exists $R_c>1$, depending only on $c$, such that
\begin{equation}\label{eq:general-short}
\frac1N\#\{x\in\Fcal:\ell(x)<c\sqrt{\log R}\}
\leq3\exp\bigl(-\kappa_c\sqrt{\log R}\bigr)\qquad(R\geq R_c).
\end{equation}
\end{enumerate}

\end{theorem}

For the fractions $a/b$, $1\leq a\leq b$, we may take $\Delta=1$ in~\eqref{eq:discrepancy}, so that
$R=b$ (see Lemma~\ref{lem:grid-discrepancy}). As further applications of
Theorem~\ref{thm:discrepancy}, we consider reduced fractions and numerators in a fixed congruence
class. For each integer $b\geq2$, we define
\begin{equation}\label{eq:primitive-mean}
Q^*(b)\coloneqq\frac1{\varphi(b)}\sum_{\substack{1\leq a\leq b\\ \gcd(a,b)=1}}\ell(a,b),
\qquad R_b\coloneqq\frac{\varphi(b)}{2^{\omega(b)}},
\qquad L_b\coloneqq\logplus R_b,
\end{equation}
where $\varphi$ is Euler's totient function and $\omega(b)$ is the number of distinct prime
factors of $b$. By M\"obius inversion, the reduced fractions with denominator $b$
satisfy~\eqref{eq:discrepancy} with $N=\varphi(b)$ and $\Delta=2^{\omega(b)}$, which leads to the
following corollary.

\begin{corollary}\label{cor:primitive}

For the reduced fractions, the following hold.
\begin{enumerate}
\item For every integer $b\geq2$, we have
\begin{equation}\label{eq:primitive-bound}
Q^*(b)\geq\sqrt{2L_b}-3L_b^{1/4}-11.
\end{equation}
In particular,
\begin{equation}\label{eq:primitive-liminf}
\liminf_{b\to\infty}\frac{Q^*(b)}{\sqrt{\log b}}\geq\sqrt2.
\end{equation}
\item For each $0<c<\sqrt2$, there exist $\gamma_c>0$ and an integer $b_c^*\geq2$ such that
\begin{equation}\label{eq:primitive-short}
\#\{1\leq a\leq b:\gcd(a,b)=1,\ \ell(a,b)<c\sqrt{\log b}\}
\leq3\varphi(b)\exp\bigl(-\gamma_c\sqrt{\log b}\bigr)
\end{equation}
for every integer $b\geq b_c^*$.
\end{enumerate}

\end{corollary}

Similarly, the numerators in a fixed congruence class within $1\leq a\leq b$ have discrepancy
at most $2$ after division by $b$. For $b,m\in\N$ and $h\in\Z$, let
\begin{equation}\label{eq:AP-set}
\mathcal A(b;m,h)\coloneqq\{1\leq a\leq b:a\equiv h\pmod m\}.
\end{equation}

\begin{corollary}\label{cor:AP}

Let $b,m\in\N$ and $h\in\Z$, and suppose that $N\coloneqq\#\mathcal A(b;m,h)$ is positive. Define
$L\coloneqq\logplus(N/2)$. Then, the following hold.
\begin{enumerate}
\item We have
\begin{equation}\label{eq:AP-mean}
\frac1N\sum_{a\in\mathcal A(b;m,h)}\ell(a,b)
\geq\sqrt{2L}-3L^{1/4}-11.
\end{equation}
\item For each $0<c<\sqrt2$, if $N/2\geq R_c$, where $R_c$ is as in
Theorem~\ref{thm:discrepancy}\textup{(ii)}, then
\begin{equation}\label{eq:AP-short}
\frac1N\#\{a\in\mathcal A(b;m,h):\ell(a,b)<c\sqrt{\log(N/2)}\}
\leq3\exp\bigl(-\kappa_c\sqrt{\log(N/2)}\bigr).
\end{equation}
\end{enumerate}

\end{corollary}

It is worth noting that the bounds in Corollary~\ref{cor:AP} depend only on the number $N$ of
numerators, and so they are uniform in the modulus $m$ and the residue class $h$. In particular, they apply
to any sequence of such progressions whose cardinalities tend to infinity.

We briefly describe the idea of the proof of Theorem~\ref{thm:discrepancy}. In the proof
of~\cite[Theorem~7]{ES91}, Erd\H{o}s and Shallit counted the occurrences of each digit in a
logarithmic range by summing over its possible preceding digits. Instead, we count the initial
blocks of digits whose product does not exceed a parameter $B$ (Lemma~\ref{lem:word-count}), and
we sum the discrepancies over the intervals determined by these blocks (Lemma~\ref{lem:transfer}).
The total length of these intervals is estimated by means of Shallit's description of the digits of
a uniformly distributed point~\cite[Theorems~1 and~2]{Sha86}, which we realize by independent
uniform random variables (see Lemmas~\ref{lem:coupling-bound} and~\ref{lem:deviations}). This
allows us to control an initial block of digits of length close to $\sqrt{2\log R}$. This
threshold arises because the logarithm of the product of the first $n$ digits of such a point is
stochastically dominated by a random variable with mean $n(n+1)/2+\log(n!)$.

This paper is organized as follows. In Section~\ref{sec:preliminaries}, we present some elementary
facts about finite Pierce expansions and the digits of a uniformly distributed point. In
Section~\ref{sec:auxiliary}, we establish some auxiliary results that will be used in the proof of
Theorem~\ref{thm:discrepancy}. In Section~\ref{sec:proofs}, we prove Theorem~\ref{thm:discrepancy},
and then deduce Theorems~\ref{thm:mean-grid} and~\ref{thm:short-grid} and
Corollaries~\ref{cor:primitive} and~\ref{cor:AP} from it. Finally, in Section~\ref{sec:identities},
we record some exact identities for the total lengths.

Throughout the paper, we denote by $\N$ the set of positive integers,
$\N_\infty\coloneqq\N\cup\{\infty\}$ the set of extended positive integers, and $\mu$ the M\"obius
function. For $b\in\Z$ and $r\in\N$, we denote by $b\bmod r$ the least non-negative residue of $b$
modulo $r$, and for $y\in\mathbb R$, we denote by $\lceil y\rceil$ the smallest integer not less
than $y$. For an interval $J\subseteq\mathbb R$, we denote by $|J|$ its length, and for a finite set
or multiset $\mathcal S$, we denote by $\#\mathcal S$ its cardinality. The logarithm $\log$ always
means the natural logarithm. All random variables are defined on a common probability space, and we
denote by $\Prob$, $\E$, and $\Var$ its probability measure, the expectation, and the variance,
respectively.

\section{Preliminaries}\label{sec:preliminaries}

This section is devoted to presenting some elementary facts about finite Pierce expansions and the
Pierce expansion digits of a uniformly distributed point. These facts are well known
(see~\cite{Ahn23,ES91,Sha86}), and we recall them in the form needed in this paper.

To generate the Pierce expansion of $x\in[0,1]$, we begin with two maps $d_1\colon[0,1]\to\N_\infty$
and $T\colon[0,1]\to[0,1]$ given by
\begin{equation}\label{eq:digit-map}
d_1(x)\coloneqq
\begin{cases}
\lfloor1/x\rfloor, & \text{if }x\neq0,\\
\infty, & \text{if }x=0,
\end{cases}
\qquad\text{and}\qquad
T(x)\coloneqq
\begin{cases}
1-d_1(x)x, & \text{if }x\neq0,\\
0, & \text{if }x=0,
\end{cases}
\end{equation}
respectively, where $\lfloor y\rfloor$ denotes the largest integer not exceeding $y\in\mathbb R$.
For each $n\in\N$, we define $d_n(x)\coloneqq d_1(T^{n-1}(x))$. Let $y\in(0,1]$. By definition, we
have $d_1(y)\leq1/y<d_1(y)+1$, and so $1/(d_1(y)+1)<y$. It follows that
\begin{equation}\label{eq:T-bound}
0<T(y)<\frac1{d_1(y)+1}\qquad\text{whenever }T(y)\neq0.
\end{equation}
Applying~\eqref{eq:T-bound} to $y=T^{n-1}(x)$, we see that $d_{n+1}(x)\geq d_n(x)+1$ for each
$x\in(0,1]$ and $n\in\N$ with $d_{n+1}(x)<\infty$. Thus, the finite digits of each $x\in(0,1]$ are
strictly increasing.

For $x\in(0,1]$ and $n\in\N$ with $d_n(x)<\infty$, we define $P_n(x)\coloneqq d_1(x)\cdots d_n(x)$.
Since $y=(1-T(y))/d_1(y)$ for each $y\in(0,1]$ by definition, we find by induction on $n$
that
\begin{equation}\label{eq:series-remainder}
x=\sum_{j=1}^{n}\frac{(-1)^{j+1}}{P_j(x)}
+\frac{(-1)^nT^n(x)}{P_n(x)}
\end{equation}
for each $x\in(0,1]$ and $n\in\N$ with $d_n(x)<\infty$. In particular, if $T^m(x)=0$ for some
$m\in\N$, then $x$ is rational. For a rational number $x\in(0,1]$, let $\ell(x)$ be the least
$m\in\N$ such that $T^m(x)=0$; such an $m$ exists by Lemma~\ref{lem:algorithm} below. Then,
$d_n(x)<\infty$ if and only if $n\leq\ell(x)$, i.e., the digit sequence of $x$ consists of $\ell(x)$
finite digits followed by $\infty$'s. We call $\ell(x)$ the length of the (finite) Pierce expansion
of $x$. On the other hand, every irrational number in $(0,1]$ has infinitely many finite digits.

The following lemma shows that, for rational numbers, the digits are computed by the remainder
iteration~\eqref{eq:intro-remainders}.

\begin{lemma}[{see~\cite[Observation~1]{ES91} and~\cite[Proposition~2.2]{Ahn23}}]\label{lem:algorithm}

Let $a$ and $b$ be integers with $1\leq a\leq b$, let $x\coloneqq a/b$, and let $(r_j)$ be as
in~\eqref{eq:intro-remainders}. Then, the following hold.
\begin{enumerate}
\item For each $0\leq j<\ell(a,b)$, we have
\begin{equation}\label{eq:remainder-map}
T^j(x)=\frac{r_j}{b}
\qquad\text{and}\qquad
d_{j+1}(x)=\left\lfloor\frac b{r_j}\right\rfloor.
\end{equation}
\item $T^{\ell(a,b)}(x)=0$, and so $\ell(x)=\ell(a,b)$.
\item We have
\begin{equation}\label{eq:finite-series}
x=\sum_{j=1}^{\ell(x)}\frac{(-1)^{j+1}}{P_j(x)}.
\end{equation}
\item If $m\coloneqq\ell(x)\geq2$, then $d_m(x)\geq d_{m-1}(x)+2$.
\end{enumerate}

\end{lemma}

\begin{proof}

Parts (i)--(iii) are essentially~\cite[Observation~1]{ES91} and its proof (see
also~\cite[p.~23]{Sha86}). Indeed, for each $1\leq r\leq b$, we have $d_1(r/b)=\lfloor b/r\rfloor$
and $T(r/b)=(b\bmod r)/b$ by~\eqref{eq:digit-map}. Hence, (i) and (ii) follow by induction on $j$,
and (iii) follows by taking $n=\ell(x)$ in~\eqref{eq:series-remainder}. Part (iv)
is~\cite[Proposition~2.2]{Ahn23} (see also~\cite[p.~44]{ES91}).
\end{proof}

Since the finite digits are strictly increasing, we have $P_n(x)\geq n!$ whenever $d_n(x)<\infty$.
Therefore, for irrational $x$, on letting $n\to\infty$ in~\eqref{eq:series-remainder}, we see
that~\eqref{eq:intro-Pierce} holds, since $0\leq T^n(x)\leq1$. For rational $x$, the same expansion
follows from~\eqref{eq:finite-series} and the convention $1/\infty=0$.

Observe that $\ell(x)$ depends only on $x$, and not on the choice of the representation $x=a/b$. We
emphasize that we always use the digits generated by~\eqref{eq:digit-map}; that is, the expansion of
a rational number is never replaced by an alternative expansion of the same number.

We now describe the points whose Pierce expansions begin with a given block of digits. A finite
sequence $\sigma=(\sigma_1,\ldots,\sigma_n)$ of positive integers with $\sigma_1<\cdots<\sigma_n$,
where $n\in\N$, is called an admissible word of length $n$. For such $\sigma$, we define
\begin{equation}\label{eq:word-definitions}
P(\sigma)\coloneqq\prod_{j=1}^{n}\sigma_j
\qquad\text{and}\qquad
S(\sigma)\coloneqq\sum_{j=1}^{n}\frac{(-1)^{j+1}}{\sigma_1\cdots\sigma_j},
\end{equation}
and let $J_\sigma$ be the open interval with endpoints
\begin{equation}\label{eq:cylinder-endpoints}
S(\sigma)\qquad\text{and}\qquad
S(\sigma)+\frac{(-1)^n}{P(\sigma)(\sigma_n+1)}.
\end{equation}
Note that the second endpoint is equal to $S(\sigma_1,\ldots,\sigma_{n-1},\sigma_n+1)$. We also
denote by $I_\sigma\coloneqq\{x\in(0,1]:d_k(x)=\sigma_k\text{ for all }1\leq k\leq n\}$ the
fundamental interval associated with $\sigma$. By~\eqref{eq:series-remainder}, every $x\in I_\sigma$
satisfies $x=S(\sigma)+(-1)^nT^n(x)/P(\sigma)$, and so $T^n(x)=0$ if and only if $x=S(\sigma)$.

\begin{lemma}[{see~\cite[Theorem~1]{Sha86} and~\cite[Proposition~3.5]{Ahn23}}]\label{lem:cylinders}

Let $\sigma$ be an admissible word of length $n$. Then, a point $x\in(0,1]$ belongs to $J_\sigma$ if
and only if $d_k(x)=\sigma_k$ for all $1\leq k\leq n$ and $T^n(x)>0$. Moreover,
\begin{equation}\label{eq:cylinder-length}
|J_\sigma|=\frac1{P(\sigma)(\sigma_n+1)},
\end{equation}
and the intervals $J_\tau$, where $\tau$ runs over the admissible words of length $n$, are pairwise
disjoint, and their lengths sum to $1$.

\end{lemma}

\begin{proof}

By~\cite[Theorem~1]{Sha86} and~\cite[Proposition~3.5]{Ahn23}, $I_\sigma$ is an interval with
endpoints~\eqref{eq:cylinder-endpoints}, and it does not contain the second endpoint. Since
$T^n(x)=0$ for $x\in I_\sigma$ if and only if $x=S(\sigma)$, the set of $x\in I_\sigma$ with
$T^n(x)>0$ is equal to $I_\sigma\setminus\{S(\sigma)\}=J_\sigma$. This proves the first assertion,
and~\eqref{eq:cylinder-length} is immediate from~\eqref{eq:cylinder-endpoints}. For the last
assertion, note that the digits of a point are uniquely determined by definition, and so the
intervals $J_\tau$, where $\tau$ runs over the admissible words of length $n$, are pairwise
disjoint. Moreover, by the first assertion, every irrational number in $(0,1)$ belongs to one of
them. As the rational numbers form a set of Lebesgue measure zero, the lengths of these intervals
sum to $1$.
\end{proof}

By Lemma~\ref{lem:cylinders}, $J_\sigma$ is the interior of $I_\sigma$. We work with the open
intervals $J_\sigma$ in order to make the conventions at rational endpoints explicit.

\begin{remark}\label{rem:endpoints}

Let $x\in(0,1]$ be rational, and let $n\in\N$. By Lemma~\ref{lem:cylinders}, if $n<\ell(x)$, then
$x$ belongs to exactly one of the intervals $J_\sigma$, where $\sigma$ runs over the admissible
words of length $n$, namely the one with $\sigma=(d_1(x),\ldots,d_n(x))$; if $n\geq\ell(x)$, then
$x$ belongs to none of them. Consequently, if $x\in J_\sigma$ for some admissible word $\sigma$ of
length $n$, then $\ell(x)>n$. These facts will be used in the proof of Lemma~\ref{lem:transfer}.

\end{remark}

Let $X$ be a random variable uniformly distributed on $(0,1)$. As $X$ is irrational almost surely,
the digits $d_n(X)$ and the products $P_n(X)$ are finite for all $n\in\N$ almost surely. We write
$d_n$ and $P_n$ for $d_n(X)$ and $P_n(X)$, respectively, whenever no confusion can arise. The
following proposition, due to Shallit, describes the digits $d_n$ as a Markov chain.

\begin{proposition}[{\cite[Theorems~1 and~2]{Sha86}}]\label{prop:digit-law}

We have
\begin{equation}\label{eq:first-digit-law}
\Prob(d_1\geq k)=\frac1k\qquad(k\in\N).
\end{equation}
Furthermore, for each $n\in\N$, we have
\begin{equation}\label{eq:transition-law}
\Prob(d_{n+1}=e\mid d_1,\ldots,d_n)
=\frac{d_n+1}{e(e+1)}\qquad(e\geq d_n+1),
\end{equation}
and, consequently,
\begin{equation}\label{eq:transition-tail}
\Prob(d_{n+1}\geq k\mid d_1,\ldots,d_n)
=\frac{d_n+1}{k}\qquad(k\geq d_n+1).
\end{equation}

\end{proposition}

Since the digits take values in the countable set $\N_\infty$, the conditional probabilities
in~\eqref{eq:transition-law} and~\eqref{eq:transition-tail} are understood in the elementary sense.
For instance, \eqref{eq:transition-law} means that, for every admissible word $\sigma$ of length $n$
and every integer $e\geq\sigma_n+1$, the conditional probability that $d_{n+1}=e$, given
$d_k=\sigma_k$ for all $1\leq k\leq n$, is equal to $(\sigma_n+1)/(e(e+1))$. The same convention
applies to the conditional probabilities in the proof of Lemma~\ref{lem:coupling-bound}.

\begin{proof}

By definition, $d_1\geq k$ if and only if $X\leq1/k$, and so~\eqref{eq:first-digit-law} holds. Now,
let $\sigma$ be an admissible word of length $n$, and let $e\geq\sigma_n+1$ be an integer. By
Lemma~\ref{lem:cylinders}, the event $\{d_k=\sigma_k\text{ for all }1\leq k\leq n\}$ differs from
$\{X\in J_\sigma\}$ by a null set, and the same holds for the word $(\sigma_1,\ldots,\sigma_n,e)$.
Hence, by~\eqref{eq:cylinder-length}, the conditional probability that $d_{n+1}=e$, given
$d_k=\sigma_k$ for all $1\leq k\leq n$, is
\[
\frac{1/(P(\sigma)e(e+1))}{1/(P(\sigma)(\sigma_n+1))}
=\frac{\sigma_n+1}{e(e+1)}.
\]
This proves~\eqref{eq:transition-law}, and~\eqref{eq:transition-tail} follows by summing
\eqref{eq:transition-law} over $e\geq k$.
\end{proof}

\section{Auxiliary results}\label{sec:auxiliary}

In this section, we establish some auxiliary results that will be used in the proofs of the main
results.

We first count admissible words with bounded product. For each real number $B\geq1$, we define
\begin{equation}\label{eq:words-B}
\Wcal(B)\coloneqq\{\sigma\text{ admissible}:P(\sigma)\leq B\}
\qquad\text{and}\qquad H(B)\coloneqq\#\Wcal(B),
\end{equation}
where admissible words of all lengths are allowed. Since $\sigma_j\geq j$ for each $j$, we have
$P(\sigma)\geq n!$ for every admissible word $\sigma$ of length $n$. In addition, every entry of a
word in $\Wcal(B)$ is at most $B$. Consequently, $\Wcal(B)$ is finite. We emphasize that the words in
$\Wcal(B)$ serve as initial blocks of digits, and they need not be the complete digit sequences of
rational numbers. For instance, $\Wcal(B)$ may contain words whose last two entries are consecutive
integers, while the last two finite digits of a rational number never are (see
Lemma~\ref{lem:algorithm}(iv)).

\begin{lemma}\label{lem:word-count}

For every real number $B\geq1$, we have
\begin{equation}\label{eq:word-count}
H(B)\leq2B\exp\bigl(2\sqrt{\log B}\bigr).
\end{equation}

\end{lemma}

\begin{proof}

If $B=1$, then $\Wcal(B)$ consists of the single word $(1)$, and~\eqref{eq:word-count} is clear.
Assume that $B>1$, and let $s>1$. We use Rankin's method~\cite{Ran38}, which is based on the simple
observation that $(B/P(\sigma))^s\geq1$ for every $\sigma\in\Wcal(B)$. It follows that
\[
H(B)\leq\sum_{\sigma\in\Wcal(B)}\Bigl(\frac B{P(\sigma)}\Bigr)^s\leq B^s\sum_{\sigma}P(\sigma)^{-s},
\]
where the last sum runs over all admissible words. Observe that an admissible word is determined by
the set of its entries, and that every non-empty finite subset of $\N$ arises in this way. Hence, on
expanding the product, we obtain
\[
\sum_{\sigma}P(\sigma)^{-s}=\prod_{j=1}^{\infty}(1+j^{-s})-1,
\]
where the product converges since $s>1$. Therefore, we have
\begin{align}\label{eq:rankin}
H(B)
&\leq B^s\prod_{j=1}^{\infty}(1+j^{-s})
=2B^s\prod_{j=2}^{\infty}(1+j^{-s})\nonumber\\
&\leq2\exp\Bigl(s\log B+\sum_{j=2}^{\infty}j^{-s}\Bigr)
\leq2\exp\Bigl(s\log B+\frac1{s-1}\Bigr),
\end{align}
where we used $1+u\leq e^u$ and
\[
\sum_{j=2}^{\infty}j^{-s}\leq\int_1^{\infty}u^{-s}\,du=\frac1{s-1}.
\]
On taking $s\coloneqq1+1/\sqrt{\log B}$ in~\eqref{eq:rankin}, we see that~\eqref{eq:word-count}
holds.
\end{proof}

Let $X$ be as in Section~\ref{sec:preliminaries}. For each real number $B\geq1$, we define
\begin{equation}\label{eq:KB}
K_B\coloneqq\#\{n\in\N:P_n(X)\leq B\}.
\end{equation}
Since $P_n(X)\geq n!$ for all $n\in\N$ almost surely, $K_B$ is bounded for each fixed $B$. The
random variable $X$ is auxiliary, and it is used only to compute the total length of the intervals
$J_\sigma$ in~\eqref{eq:KB-sum} below. In other words, we do not assume that
the digits of rational numbers are independent or follow the distribution of the digits of $X$.

The following lemma relates the lengths of the elements of a finite multiset of rational numbers to
the distribution of the products $P_n(X)$.

\begin{lemma}\label{lem:transfer}

Let $\Fcal$, $N$, $\Delta$, and $R$ be as in Theorem~\ref{thm:discrepancy}, and let $B\geq1$ be a
real number. Then, the following hold.
\begin{enumerate}
\item We have
\begin{equation}\label{eq:transfer-mean}
\frac1N\sum_{x\in\Fcal}\ell(x)
\geq1+\E K_B-\frac{H(B)}R.
\end{equation}
\item For every $n\in\N$, we have
\begin{equation}\label{eq:transfer-short}
\frac1N\#\{x\in\Fcal:\ell(x)<n\}
\leq\Prob(P_n>B)+\frac{H(B)}R.
\end{equation}
\end{enumerate}

\end{lemma}

\begin{proof}

(i) Let $\sigma\in\Wcal(B)$. As $J_\sigma$ is an open subinterval of $(0,1)$, it follows
from~\eqref{eq:discrepancy} that
\[
\#(\Fcal\cap J_\sigma)\geq N|J_\sigma|-\Delta.
\]
By Remark~\ref{rem:endpoints}, for each $n\in\N$, every $x\in\Fcal$ lies in at most one interval
$J_\sigma$ such that $\sigma\in\Wcal(B)$ has length $n$, and in no such interval if $n\geq\ell(x)$.
Thus, each $x\in\Fcal$ lies in at most $\ell(x)-1$ of the intervals $J_\sigma$, $\sigma\in\Wcal(B)$.
Summing over $\sigma\in\Wcal(B)$, we find that
\[
\sum_{x\in\Fcal}\bigl(\ell(x)-1\bigr)
\geq\sum_{\sigma\in\Wcal(B)}\#(\Fcal\cap J_\sigma)
\geq N\sum_{\sigma\in\Wcal(B)}|J_\sigma|-\Delta H(B).
\]
On the other hand, by Lemma~\ref{lem:cylinders} and the definition~\eqref{eq:KB} of $K_B$, we have
\begin{equation}\label{eq:KB-sum}
\sum_{\sigma\in\Wcal(B)}|J_\sigma|
=\sum_{n\in\N}\Prob(P_n\leq B)=\E K_B.
\end{equation}
Since $\sum_{x\in\Fcal}(\ell(x)-1)=\sum_{x\in\Fcal}\ell(x)-N$ and $\Delta/N=1/R$, the
inequality~\eqref{eq:transfer-mean} follows.

(ii) Let $n\in\N$, and consider the words in $\Wcal(B)$ of length $n$. By Lemma~\ref{lem:cylinders},
the corresponding intervals $J_\sigma$ are pairwise disjoint, and their union has measure
$\Prob(P_n\leq B)$. Furthermore, by Remark~\ref{rem:endpoints}, every rational number in their union
has length greater than $n$, and there are at most $H(B)$ such intervals. Summing the discrepancy
estimates over them, we have
\[
\#\{x\in\Fcal:\ell(x)\geq n\}
\geq N\Prob(P_n\leq B)-\Delta H(B),
\]
which is equivalent to~\eqref{eq:transfer-short}. This completes the proof of the lemma.
\end{proof}

In order to estimate $\E K_B$ and $\Prob(P_n>B)$, we need tail bounds for the products $P_n$ with
explicit constants. Shallit obtained asymptotic formulas for the mean and the variance of
$\log d_n$~\cite[Theorems~9 and~10]{Sha86}, and described the set of digits of $X$ by means of
independent indicator random variables~\cite[Section~5]{Sha86}. For our purposes, it is more
convenient to realize the digits of $X$ by means of independent uniform random variables. Let
$U_1,U_2,\ldots$ be independent random variables uniformly distributed on $(0,1)$, and define
\begin{equation}\label{eq:uniform-coupling}
D_1\coloneqq\lfloor U_1^{-1}\rfloor
\qquad\text{and}\qquad
D_{n+1}\coloneqq\left\lfloor\frac{D_n+1}{U_{n+1}}\right\rfloor\quad(n\in\N).
\end{equation}
For each $n\in\N$, let $Y_n\coloneqq-\log U_n$. Then, $Y_1,Y_2,\ldots$ are independent and
exponentially distributed with mean $1$ and variance $1$.

\begin{lemma}\label{lem:coupling-bound}

The sequence $(D_n)_{n\in\N}$ has the same distribution as $(d_n(X))_{n\in\N}$. Moreover, for each
$n\in\N$, we have
\begin{equation}\label{eq:digit-coupling-bound}
\log D_n\leq\sum_{i=1}^{n}Y_i+\log n
\end{equation}
and
\begin{equation}\label{eq:product-coupling-bound}
\log(D_1\cdots D_n)\leq W_n+\log(n!),
\qquad\text{where }W_n\coloneqq\sum_{j=1}^{n}(n-j+1)Y_j.
\end{equation}

\end{lemma}

\begin{proof}

Clearly, $\Prob(D_1\geq k)=\Prob(U_1\leq1/k)=1/k$ for each $k\in\N$. Let $n\in\N$. As $U_{n+1}$ is
independent of $D_1,\ldots,D_n$, we have, for each integer $k\geq D_n+1$,
\[
\Prob(D_{n+1}\geq k\mid D_1,\ldots,D_n)
=\Prob\Bigl(U_{n+1}\leq\frac{D_n+1}k\Bigm|D_1,\ldots,D_n\Bigr)
=\frac{D_n+1}k.
\]
Since $D_{n+1}\geq D_n+1$ by definition, this determines the conditional distribution of $D_{n+1}$.
Hence, in view of Proposition~\ref{prop:digit-law}, the sequences $(D_n)$ and $(d_n(X))$ are Markov
chains with the same initial distribution and the same transition probabilities, which proves the
first assertion.

For~\eqref{eq:digit-coupling-bound}, note first that $D_1\leq U_1^{-1}$, which gives the case $n=1$.
Since $D_1\geq1$ and $D_{k+1}\geq D_k+1$ for each $k\in\N$, we have $D_k\geq k$ for each $k\in\N$.
Thus, for each $n\geq2$,
\begin{align*}
\log D_n
&\leq\log(D_{n-1}+1)+Y_n
=\log D_{n-1}+\log\Bigl(1+\frac1{D_{n-1}}\Bigr)+Y_n\\
&\leq\log D_{n-1}+Y_n+\log\frac n{n-1}.
\end{align*}
Iterating this inequality, we find that~\eqref{eq:digit-coupling-bound} holds. Finally, summing
\eqref{eq:digit-coupling-bound}, with $n$ replaced by $k$, over $1\leq k\leq n$, we
deduce~\eqref{eq:product-coupling-bound}.
\end{proof}

By Lemma~\ref{lem:coupling-bound}, the product $P_n=P_n(X)$ has the same distribution as
$D_1\cdots D_n$ for each $n\in\N$. For each $n\in\N$, we define
\begin{equation}\label{eq:MV}
M_n\coloneqq\frac{n(n+1)}2+\log(n!)
\qquad\text{and}\qquad
V_n\coloneqq\frac{n(n+1)(2n+1)}6.
\end{equation}
Note that $\E W_n=n(n+1)/2$ and $\Var W_n=V_n$ by the independence of the $Y_j$, and hence
$M_n=\E W_n+\log(n!)$.

\begin{lemma}\label{lem:deviations}

Let $n\in\N$ and $B>1$, and suppose that $t\coloneqq\log B-M_n>0$. Then, we have
\begin{align}
\Prob(P_n>B)&\leq\min\Bigl\{1,\frac{V_n}{t^2}\Bigr\},\label{eq:chebyshev}\\
\Prob(P_n>B)&\leq\exp\Bigl(-\frac{t^2}{2(V_n+nt)}\Bigr).\label{eq:exponential}
\end{align}

\end{lemma}

\begin{proof}

By Lemma~\ref{lem:coupling-bound}, we may replace $P_n$ by $D_1\cdots D_n$. Then,
by~\eqref{eq:product-coupling-bound}, the event $\{P_n>B\}$ is contained in the event
$\{W_n-\E W_n>t\}$. Hence, \eqref{eq:chebyshev} follows from Chebyshev's inequality.

For~\eqref{eq:exponential}, let $0<\theta<1/n$. Recall that $\E e^{uY_1}=(1-u)^{-1}$ for
$0\leq u<1$. By the independence of the $Y_j$, it follows that
\begin{align}\label{eq:moment-bound}
\log\E\exp\bigl(\theta(W_n-\E W_n)\bigr)
&=\sum_{k=1}^{n}\bigl(-\log(1-\theta k)-\theta k\bigr)\nonumber\\
&\leq\sum_{k=1}^{n}\frac{\theta^2k^2}{2(1-\theta k)}
\leq\frac{\theta^2V_n}{2(1-\theta n)},
\end{align}
where we used
\[
-\log(1-u)-u=\sum_{j=2}^{\infty}\frac{u^j}j
\leq\frac12\sum_{j=2}^{\infty}u^j=\frac{u^2}{2(1-u)}\qquad(0\leq u<1).
\]
Therefore, by Markov's inequality applied to $\exp(\theta(W_n-\E W_n))$, we have
\[
\Prob(P_n>B)\leq\exp\Bigl(-\theta t+\frac{\theta^2V_n}{2(1-\theta n)}\Bigr).
\]
Now, put $\theta\coloneqq t/(V_n+nt)$. Then, $0<\theta<1/n$ and $1-\theta n=V_n/(V_n+nt)$, and so the
exponent on the right-hand side is equal to $-t^2/(2(V_n+nt))$. This proves~\eqref{eq:exponential}.
\end{proof}

Combining Lemmas~\ref{lem:word-count}, \ref{lem:transfer}(ii), and~\ref{lem:deviations}, we obtain
the following bound with two free parameters $B$ and $n$.

\begin{lemma}\label{lem:finite-bound}

Let $\Fcal$, $N$, $\Delta$, and $R$ be as in Theorem~\ref{thm:discrepancy}. Let $n\in\N$ and
$B>1$, and suppose that $t\coloneqq\log B-M_n>0$. Then, we have
\begin{equation}\label{eq:finite-bound}
\frac1N\#\{x\in\Fcal:\ell(x)<n\}
\leq\exp\Bigl(-\frac{t^2}{2(V_n+nt)}\Bigr)
+\frac{2B\exp(2\sqrt{\log B})}{R}.
\end{equation}

\end{lemma}

\begin{proof}

It suffices to bound the two terms on the right-hand side of~\eqref{eq:transfer-short}
by~\eqref{eq:exponential} and~\eqref{eq:word-count}, respectively.
\end{proof}

Lastly, we record the elementary discrepancy bound for the fractions with a fixed denominator.

\begin{lemma}\label{lem:grid-discrepancy}

Let $b\in\N$. Then, for all $0\leq u<v\leq1$, we have
\begin{equation}\label{eq:grid-count}
\bigl|\#\{1\leq a\leq b:u<a/b<v\}-b(v-u)\bigr|\leq1.
\end{equation}

\end{lemma}

\begin{proof}

The count on the left-hand side of~\eqref{eq:grid-count} is the number of integers in the open
interval $(bu,bv)$, which is equal to $\lceil bv\rceil-\lfloor bu\rfloor-1$. As
$0\leq\lceil bv\rceil-bv<1$ and $0\leq bu-\lfloor bu\rfloor<1$, this number differs from $b(v-u)$
by at most $1$, as desired.
\end{proof}

\section{Proofs of main results}\label{sec:proofs}

In this section, we prove the main results stated in Section~\ref{sec:introduction}. We first prove
Theorem~\ref{thm:discrepancy} in Section~\ref{subsec:proof-general}, and then deduce the other
results from it in Section~\ref{subsec:proof-others}.

\subsection{Proof of Theorem~\ref{thm:discrepancy}}\label{subsec:proof-general}

Throughout this subsection, let $\Fcal$, $N$, $\Delta$, $R$, and $L$ be as in
Theorem~\ref{thm:discrepancy}. We prove parts (i) and (ii) of Theorem~\ref{thm:discrepancy}
separately.

\begin{proof}[Proof of Theorem~\ref{thm:discrepancy}\textup{(i)}]

If $L<9$, then the right-hand side of~\eqref{eq:general-mean} is negative, and there is nothing to
prove. We may therefore assume that $L\geq9$, so that $R=e^L$. We define
\begin{equation}\label{eq:choice-Bm}
B\coloneqq e^{L-3\sqrt L},\qquad \tau_B\coloneqq\sqrt{2\log B},\qquad
m\coloneqq\max\bigl\{0,\lfloor\tau_B-2\log(\tau_B+1)-3\rfloor\bigr\}.
\end{equation}
Clearly, $B\geq1$, since $L\geq9$. By Lemma~\ref{lem:word-count}, we have
\begin{equation}\label{eq:word-error}
\frac{H(B)}R
\leq\frac{2B\exp(2\sqrt{\log B})}{R}
=2\exp\bigl(-3\sqrt L+2\sqrt{L-3\sqrt L}\bigr)
\leq2e^{-\sqrt L}.
\end{equation}

Assume first that $m\geq1$. We show that $M_m\leq\log B$. As $\log B=\tau_B^2/2$ and
$m\leq\tau_B-2\log(\tau_B+1)-3$, we have
\[
\log B-\frac{m^2}2
=\frac{(\tau_B-m)(\tau_B+m)}2
\geq m\bigl(2\log(\tau_B+1)+3\bigr)
\geq m\log m+\frac m2.
\]
Since $\log(m!)\leq m\log m$, it follows that $M_m\leq\log B$. Now, let $1\leq n<m$, and put
$j\coloneqq m-n$. Then, we have
\begin{equation}\label{eq:distance-from-Mn}
\log B-M_n\geq M_m-M_n
=\sum_{k=n+1}^{m}(k+\log k)\geq j(n+1).
\end{equation}
Since $V_n\leq n(n+1)^2/3$, it follows from~\eqref{eq:chebyshev} that
\begin{equation}\label{eq:mean-tail}
\Prob(P_n>B)\leq\min\Bigl\{1,\frac{n(n+1)^2}{3j^2(n+1)^2}\Bigr\}
\leq\min\Bigl\{1,\frac m{3j^2}\Bigr\}.
\end{equation}
For $n=m$, we use the trivial bound $\Prob(P_m>B)\leq1$. Since the function
$u\mapsto\min\{1,m/(3u^2)\}$ is non-increasing on $(0,\infty)$, we find that
\begin{align}\label{eq:sum-mean-tails}
\sum_{n=1}^{m}\Prob(P_n>B)
&\leq1+\sum_{j=1}^{\infty}\min\Bigl\{1,\frac m{3j^2}\Bigr\}\nonumber\\
&\leq1+\int_0^{\infty}\min\Bigl\{1,\frac m{3u^2}\Bigr\}\,du
=1+2\sqrt{m/3},
\end{align}
where the integral is evaluated by splitting it at $u=\sqrt{m/3}$. Consequently,
\[
\E K_B\geq\sum_{n=1}^{m}\Prob(P_n\leq B)
\geq m-1-2\sqrt{m/3}.
\]
This inequality holds trivially when $m=0$. Therefore, by Lemma~\ref{lem:transfer}(i)
and~\eqref{eq:word-error}, we have
\begin{equation}\label{eq:mean-intermediate}
\frac1N\sum_{x\in\Fcal}\ell(x)
\geq m-2\sqrt{m/3}-2e^{-\sqrt L}.
\end{equation}

Put $\tau\coloneqq\sqrt{2L}$. We claim that
\begin{equation}\label{eq:T-estimates}
\tau_B\geq\tau-3\sqrt2,
\qquad m\geq\tau_B-2\log(\tau_B+1)-4,
\qquad m\leq\tau_B\leq\tau.
\end{equation}
Indeed, the first inequality follows from
\[
\tau-\tau_B=\frac{\tau^2-\tau_B^2}{\tau+\tau_B}=\frac{6\sqrt L}{\tau+\tau_B}
\leq\frac{6\sqrt L}{\sqrt{2L}}=3\sqrt2,
\]
and the other inequalities are immediate from~\eqref{eq:choice-Bm}. By~\eqref{eq:T-estimates}, we
have $\sqrt{m/3}\leq\sqrt{\tau/3}$ and $m\geq\tau-3\sqrt2-2\log(\tau+1)-4$. Hence, the right-hand
side of~\eqref{eq:mean-intermediate} is at least
\begin{equation}\label{eq:mean-before-constants}
\tau-2\log(\tau+1)-\frac2{\sqrt3}\sqrt\tau
-3\sqrt2-4-2e^{-\sqrt L}.
\end{equation}
As $L\geq9$, we have $(\tau+1)^2\leq8L$ and $e^{-\sqrt L}\leq e^{-3}$. Furthermore,
$\log L\leq(4/e)L^{1/4}$, which follows by maximizing $(\log u)/u^{1/4}$ over $u\geq1$. Thus,
\[
2\log(\tau+1)\leq\log L+\log8\leq\frac4eL^{1/4}+\log8
\qquad\text{and}\qquad
\frac2{\sqrt3}\sqrt\tau=\frac{2\cdot2^{1/4}}{\sqrt3}L^{1/4}.
\]
Note that
\[
\frac{2\cdot2^{1/4}}{\sqrt3}+\frac4e<3
\qquad\text{and}\qquad
3\sqrt2+4+\log8+2e^{-3}<11.
\]
Substituting these estimates into~\eqref{eq:mean-before-constants}, we conclude
that~\eqref{eq:general-mean} holds.
\end{proof}

\begin{proof}[Proof of Theorem~\ref{thm:discrepancy}\textup{(ii)}]

Fix $0<c<\sqrt2$, and put $\delta\coloneqq1-c^2/2>0$. Suppose that $R\geq e^9$, i.e.,
$L=\log R\geq9$, and let $B$ be as in~\eqref{eq:choice-Bm}. We define
\[
n\coloneqq\lceil c\sqrt L\rceil
\qquad\text{and}\qquad
t\coloneqq\log B-M_n.
\]
We first show that
\begin{equation}\label{eq:positive-t}
t\geq\frac\delta2L
\end{equation}
for all sufficiently large $L$. As $n\leq c\sqrt L+1$ and $\log(n!)\leq n\log n$, we have
\[
M_n\leq\frac{(c\sqrt L+1)(c\sqrt L+2)}2+(c\sqrt L+1)\log(c\sqrt L+1),
\]
and hence
\[
t\geq\delta L-\Bigl(3+\frac{3c}2\Bigr)\sqrt L-1-(c\sqrt L+1)\log(c\sqrt L+1).
\]
After division by $L$, the terms subtracted from $\delta L$ tend to $0$ as $L\to\infty$, which
proves~\eqref{eq:positive-t}.

Note also that $n\leq(c+1)\sqrt L$, $V_n\leq n^3$, and $t\leq\log B\leq L$, where $V_n\leq n^3$
follows from $6n^3-n(n+1)(2n+1)=n(n-1)(4n+1)\geq0$. It follows that
\begin{equation}\label{eq:variance-upper}
V_n+nt\leq\bigl((c+1)^3+(c+1)\bigr)L^{3/2}.
\end{equation}
Now, suppose that~\eqref{eq:positive-t} holds. Then, $t>0$, and so, by Lemma~\ref{lem:finite-bound},
\eqref{eq:positive-t}, \eqref{eq:variance-upper}, and~\eqref{eq:word-error}, we have
\begin{align*}
\frac1N\#\{x\in\Fcal:\ell(x)<n\}
&\leq\exp\Bigl(-\frac{\delta^2\sqrt L}{8\bigl((c+1)^3+(c+1)\bigr)}\Bigr)+2e^{-\sqrt L}\\
&=e^{-\kappa_c\sqrt L}+2e^{-\sqrt L}
\leq3e^{-\kappa_c\sqrt L},
\end{align*}
where the last inequality follows from $\kappa_c<1$. Lastly, if an integer is smaller than
$c\sqrt L$, then it is smaller than $\lceil c\sqrt L\rceil=n$. Therefore,~\eqref{eq:general-short}
holds for all sufficiently large $R$, where the threshold depends only on $c$. This completes the
proof.
\end{proof}

\subsection{Proofs of the other results}\label{subsec:proof-others}

\begin{proof}[Proofs of Theorems~\ref{thm:mean-grid} and~\ref{thm:short-grid}]

Let $b\geq2$ be an integer, and let $\Fcal\coloneqq\{a/b:1\leq a\leq b\}$. Then, $N=b$,
and~\eqref{eq:discrepancy} holds with $\Delta=1$ by Lemma~\ref{lem:grid-discrepancy}. Hence, $R=b$
and $L=\log b$. Since $\ell(a/b)=\ell(a,b)$ by Lemma~\ref{lem:algorithm}(ii), part (i) of
Theorem~\ref{thm:discrepancy} gives~\eqref{eq:main-mean}, and part (ii) gives~\eqref{eq:short-grid}
with $b_c\coloneqq\max\{2,\lceil R_c\rceil\}$. Finally, on dividing~\eqref{eq:main-mean} by
$\sqrt{\log b}$ and letting $b\to\infty$, we see that~\eqref{eq:mean-liminf} holds.
\end{proof}

\begin{proof}[Proof of Corollary~\ref{cor:primitive}]

(i) Let $b\geq2$ be an integer, and let $0\leq u<v\leq1$ be arbitrary. By M\"obius inversion, we
have
\begin{align*}
&\#\{1\leq a\leq b:\gcd(a,b)=1,\ u<a/b<v\}\\
&\hspace{2cm}=\sum_{d\mid b}\mu(d)\,
\#\{1\leq a\leq b:d\mid a,\ u<a/b<v\}.
\end{align*}
For each $d\mid b$, on writing $a=da'$ and applying Lemma~\ref{lem:grid-discrepancy} with $b/d$ in
place of $b$, we see that the count on the right-hand side differs from $(b/d)(v-u)$ by at most
$1$. Recall that
\[
\sum_{d\mid b}\mu(d)\frac bd=\varphi(b)
\qquad\text{and}\qquad\sum_{d\mid b}|\mu(d)|=2^{\omega(b)}.
\]
Thus, the multiset $\{a/b:1\leq a\leq b,\ \gcd(a,b)=1\}$ satisfies~\eqref{eq:discrepancy} with
$N=\varphi(b)$ and $\Delta=2^{\omega(b)}$. Hence, $R=R_b$, and~\eqref{eq:primitive-bound} follows
from Theorem~\ref{thm:discrepancy}(i).

Next, we show that
\begin{equation}\label{eq:effective-size}
\lim_{b\to\infty}\frac{\log R_b}{\log b}=1.
\end{equation}
Let $\eta>0$. Since $2\leq p^\eta$ and $(1-1/p)^{-1}\leq p^\eta$ for all sufficiently large primes
$p$, by separating the finitely many remaining primes in the products
$2^{\omega(b)}=\prod_{p\mid b}2$ and $b/\varphi(b)=\prod_{p\mid b}(1-1/p)^{-1}$, we can find
constants $C_\eta,C_\eta'>0$, depending only on $\eta$, such that
\[
2^{\omega(b)}\leq C_\eta b^\eta
\qquad\text{and}\qquad\frac b{\varphi(b)}\leq C_\eta'b^\eta
\]
for all $b\in\N$. It follows that $(C_\eta C_\eta')^{-1}b^{1-2\eta}\leq R_b\leq b$,
and~\eqref{eq:effective-size} holds, since $\eta>0$ is arbitrary. Combining~\eqref{eq:effective-size}
with~\eqref{eq:primitive-bound}, we deduce~\eqref{eq:primitive-liminf}.

(ii) Let $0<c<\sqrt2$, and put $c'\coloneqq(c+\sqrt2)/2$, so that $c<c'<\sqrt2$.
By~\eqref{eq:effective-size}, we can find an integer $b_c^*\geq2$ such that, for every integer
$b\geq b_c^*$, we have $R_b\geq R_{c'}$,
\[
c\sqrt{\log b}\leq c'\sqrt{\log R_b},
\qquad\text{and}\qquad
\sqrt{\log R_b}\geq\frac12\sqrt{\log b}.
\]
Hence, by applying Theorem~\ref{thm:discrepancy}(ii) with $c'$ in place of $c$ to the multiset in
part (i), we see that~\eqref{eq:primitive-short} holds with $\gamma_c\coloneqq\kappa_{c'}/2$. This
completes the proof.
\end{proof}

\begin{proof}[Proof of Corollary~\ref{cor:AP}]

Let $0\leq u<v\leq1$ be arbitrary. On writing $a=h+mk$ with $k\in\Z$, we see that the elements of
$\mathcal A(b;m,h)$ in the open interval $(bu,bv)$ correspond to the integers $k$ in an open
interval of length $b(v-u)/m$. Note that the number of integers in an open interval of length
$\lambda$ differs from $\lambda$ by at most $1$. Hence,
\[
\Bigl|\#\{a\in\mathcal A(b;m,h):u<a/b<v\}
-\frac bm(v-u)\Bigr|\leq1.
\]
Similarly, counting the elements of $\mathcal A(b;m,h)$ in $(0,b]$, we find that $|N-b/m|\leq1$.
Consequently, the multiset $\{a/b:a\in\mathcal A(b;m,h)\}$ satisfies~\eqref{eq:discrepancy} with
$\Delta=2$. Therefore, $R=N/2$, and both parts of the corollary follow from
Theorem~\ref{thm:discrepancy}.
\end{proof}

\section{Further remarks}\label{sec:identities}

In this section, we record some exact identities for the total lengths of finite Pierce expansions.
All of them are elementary consequences of the iteration~\eqref{eq:intro-remainders}. Among others,
Proposition~\ref{prop:divisor-decomposition} relates the averages $Q(b)$ and $Q^*(b)$.

For each $b\in\N$, we define
\begin{equation}\label{eq:AC}
A(b)\coloneqq\sum_{a=1}^{b}\ell(a,b)
\qquad\text{and}\qquad
C(b)\coloneqq\sum_{\substack{1\leq a\leq b\\ \gcd(a,b)=1}}\ell(a,b).
\end{equation}
By definition, we have $A(1)=C(1)=1$, and $A(b)=bQ(b)$ and $C(b)=\varphi(b)Q^*(b)$ for $b\geq2$.

\begin{proposition}\label{prop:divisor-decomposition}

The following hold.
\begin{enumerate}
\item For any $a,b,k\in\N$ with $a\leq b$, we have
\begin{equation}\label{eq:scaling}
\ell(ka,kb)=\ell(a,b).
\end{equation}
\item For every $b\in\N$, we have
\begin{equation}\label{eq:divisor-decomposition}
A(b)=\sum_{d\mid b}C(d)
\qquad\text{and}\qquad
C(b)=\sum_{d\mid b}\mu(b/d)A(d).
\end{equation}
\end{enumerate}

\end{proposition}

\begin{proof}

(i) By induction, each remainder in the iteration~\eqref{eq:intro-remainders} for $(ka,kb)$ is $k$
times the corresponding remainder for $(a,b)$, which proves~\eqref{eq:scaling}.

(ii) For the first identity in~\eqref{eq:divisor-decomposition}, we classify the numerators
$1\leq a\leq b$ according to the reduced denominator $d\coloneqq b/\gcd(a,b)$ of $a/b$. For each
$d\mid b$, the map $a\mapsto a/(b/d)$ is a bijection from $\{1\leq a\leq b:\gcd(a,b)=b/d\}$ onto
$\{1\leq a'\leq d:\gcd(a',d)=1\}$, and $\ell(a,b)=\ell(a/(b/d),d)$ by part (i). Thus, the numerators
with reduced denominator $d$ contribute $C(d)$ to $A(b)$, which proves the first identity. The
second identity follows from the first one by M\"obius inversion.
\end{proof}

Note that~\eqref{eq:divisor-decomposition} separates the contribution of the numerators which are
not coprime to $b$ from that of the coprime numerators. However, it does not by itself determine the
asymptotic behavior of either average.

Let $b\in\N$. For each $1\leq r\leq b$, let $\nu_b(r)$ denote the number of numerators
$a\in\{1,\ldots,b\}$ whose orbit under~\eqref{eq:intro-remainders} visits $r$, i.e., $r_j=r$ for some
$j$. As the positive remainders are strictly decreasing, each orbit visits $r$ at most once. For
each $j\in\N$, let $T_b(j)$ be the total number of occurrences of the digit $j$ in the Pierce
expansions of $a/b$, $1\leq a\leq b$, as in the proof of~\cite[Theorem~7]{ES91}. Then, we have
\begin{equation}\label{eq:digit-remainder-count}
T_b(j)=\sum_{\substack{1\leq r\leq b\\ b/(j+1)<r\leq b/j}}\nu_b(r),
\end{equation}
since, by Lemma~\ref{lem:algorithm}(i), the digit produced from a positive remainder $r$ is
$\lfloor b/r\rfloor$.

\begin{proposition}\label{prop:ancestors}

Let $b\in\N$. Then, for each $1\leq r\leq b$, we have
\begin{equation}\label{eq:ancestor-recurrence}
\nu_b(r)=1+\sum_{\substack{r<s\leq b\\ s\mid b-r}}\nu_b(s).
\end{equation}
Moreover, $A(b)=\sum_{r=1}^{b}\nu_b(r)$.

\end{proposition}

\begin{proof}

Let $1\leq r\leq b$. The orbit starting from $a=r$ visits $r$, and it contributes $1$ to $\nu_b(r)$.
Every other orbit visiting $r$ has a unique immediate predecessor $s$ of $r$, i.e., a remainder $s$
with $r<s\leq b$ and $b\bmod s=r$. For $s>r$, the condition $b\bmod s=r$ is equivalent to
$s\mid b-r$ (cf.\ \cite[Lemma~4]{BV26}). Conversely, every orbit visiting such an $s$ visits $r$ at
the next step. Furthermore, no orbit visits two distinct immediate predecessors $s>s'$ of $r$.
Indeed, after visiting $s$, the orbit visits $r<s'$ at the next step, and the remainders are
strictly decreasing. Hence, the sets of numerators whose orbits visit the various immediate
predecessors of $r$ are pairwise disjoint, which proves~\eqref{eq:ancestor-recurrence}.

For the last identity, we count in two ways the pairs $(a,r)$ such that $r$ is a positive remainder
in the orbit of $a$. Since the orbit of $a$ contains exactly $\ell(a,b)$ positive remainders, the
number of such pairs is equal to $A(b)$. On the other hand, by definition, it is also equal to
$\sum_{r=1}^{b}\nu_b(r)$. This completes the proof.
\end{proof}


\begin{thebibliography}{99}

\bibitem{Ahn23}
M.~W.~Ahn,
\emph{On the error-sum function of Pierce expansions},
J. Fractal Geom. \textbf{10} (2023), 389--421.

\bibitem{BGJK26}
O.~Baraskar, P.~Gokhale, S.~Jain, and A.~Kie\.{z}un,
\emph{Asymptotic behavior of iterated sets of remainders},
preprint, 2026, \href{https://arxiv.org/abs/2609.24412}{arXiv:2609.24412}.

\bibitem{BV26}
O.~Baraskar and I.~Vukusic,
\emph{Bounds for sets of remainders},
J. Integer Seq. \textbf{29} (2026), Article 26.1.4, 18 pp.

\bibitem{CP22}
Z.~Chase and M.~Pandey,
\emph{On the length of Pierce expansions},
preprint, 2022, \href{https://arxiv.org/abs/2211.08374}{arXiv:2211.08374}.

\bibitem{ES91}
P.~Erd\H{o}s and J.~O.~Shallit,
\emph{New bounds on the length of finite Pierce and Engel series},
J. Th\'eor. Nombres Bordeaux \textbf{3} (1991), 43--53.

\bibitem{Kes96}
V.~Ke\v{s}elj,
\emph{Length of finite Pierce series: theoretical analysis and numerical computations},
Technical Report CS-96-21, University of Waterloo, 1996.

\bibitem{Mays87}
M.~E.~Mays,
\emph{Iterating the division algorithm},
Fibonacci Quart. \textbf{25} (1987), 204--213.

\bibitem{Pie29}
T.~A.~Pierce,
\emph{On an algorithm and its use in approximating roots of algebraic equations},
Amer. Math. Monthly \textbf{36} (1929), 523--525.

\bibitem{Ran38}
R.~A.~Rankin,
\emph{The difference between consecutive prime numbers},
J. London Math. Soc. \textbf{13} (1938), 242--247.

\bibitem{Sha86}
J.~O.~Shallit,
\emph{Metric theory of Pierce expansions},
Fibonacci Quart. \textbf{24} (1986), 22--40.

\end{thebibliography}
\end{document}